\documentclass[12pt]{amsart}

\usepackage{amssymb}

\newtheorem{theorem}{Theorem}[section]
\newtheorem{lemma}[theorem]{Lemma}
\newtheorem{proposition}[theorem]{Proposition}

\begin{document}
\baselineskip=15.5pt

\title[Semistability of invariant bundles over
$G/\Gamma$]{Semistability of invariant bundles over $G/\Gamma$}

\author[I. Biswas]{Indranil Biswas}

\address{School of Mathematics, Tata Institute of Fundamental
Research, Homi Bhabha Road, Bombay 400005, India}

\email{indranil@math.tifr.res.in}


\date{}

\begin{abstract}
Let $G$ be a connected reductive affine algebraic group defined over
$\mathbb C$, and let $\Gamma$ be a cocompact lattice in $G$. We prove
that any invariant bundle on $G/\Gamma$ is semistable.
\smallskip

\noindent
\textsc{R\'esum\'e.} \textbf{Semi-stabilit\'e de
fibr\'es invariants sur $G/\Gamma$.}
Soit $\Gamma$ un sous-groupe discret cocompact d'un groupe 
alg\'ebique r\'eductif affine $G$. Nous d\'emontrons que tout 
fibr\'e invariant sur $G/\Gamma$ est semi-stable.
\end{abstract}

\maketitle

\section{Introduction}\label{sec1}

Let $G$ be a connected complex reductive affine algebraic group,
and let $K\, \subset\, G$ be a maximal compact subgroup. Fixing a
$K$--invariant Hermitian form on $\text{Lie}(G)$, we may extend
it to a right--translation invariant Hermitian structure on $G$.
If $\omega_G$ is the corresponding $(1\, ,1)$--form on $G$, and
$\dim_{\mathbb C} G\,=\, \delta$, we prove that the form
$\omega^{\delta-1}_G$ is closed (see Proposition \ref{prop1}).

Let $\Gamma\, \subset\, G$ be a cocompact lattice. The descent
of $\omega_G$ to the compact quotient $G/\Gamma$ will be denoted
by $\widetilde{\omega}$. So, $d\widetilde{\omega}^{\delta-1}
\,=\, 0$ by Proposition \ref{prop1}.
This allows us to define the degree of a coherent analytic 
sheaf on $G/\Gamma$; as a consequence, semistable vector bundles 
on $G/\Gamma$ can be defined.

A vector bundle $E$ on $G/\Gamma$ is called invariant if the
pullback of $E$ using the left--translation by any element of
$G$ is holomorphically isomorphic to $E$. We prove that
invariant vector bundles are semistable (see Lemma \ref{lem1}).
It may be mentioned that Lemma \ref{lem1} remains valid for
holomorphic principal bundles on $G/\Gamma$ with a reductive
group as the structure group.

\section{Hermitian structure and semistability}

Let $G$ be a connected reductive affine algebraic group defined over 
$\mathbb C$. The Lie algebra of $G$ will be denoted
by $\mathfrak g$. Fix a maximal compact subgroup $K\, \subset\, G$.

The group $G$ has the adjoint action on $\mathfrak g$. Let $h_0$ be
a $K$--invariant inner product on the complex vector space $\mathfrak
g$. Let $h_G$ be the unique Hermitian structure on $G$, invariant under 
the right translation action of $G$ on itself, with $h_G(e)\,=\, h_0$. 
Let $\omega_G$ be the $(1\, ,1)$--form on $G$ associated to $h_G$.

\begin{proposition}\label{prop1}
Let $\delta$ be the complex dimension of $G$. Then
$$
d\omega^{\delta-1}_G\, =\, 0\, ,
$$
where $\omega_G$ is defined above.
\end{proposition}

\begin{proof}
We will first reduce it to the case of semisimple groups. Let
$Z_G$ be the connected component of the center of $G$ containing
the identity element $e$. The Lie 
algebra of $Z_G$ will be denoted by ${\mathfrak z}_{\mathfrak g}$. 
The Lie algebra ${\mathfrak g}$ decomposes as
\begin{equation}\label{e1}
{\mathfrak g}\,=\, [{\mathfrak g}\, ,{\mathfrak g}]\oplus
{\mathfrak z}_{\mathfrak g}\, ;
\end{equation}
we note that $[{\mathfrak g}\, ,{\mathfrak g}]$ is semisimple. Using
$h_0$, any element $v\, \in\, {\mathfrak z}_{\mathfrak g}$ produces
an element $\widetilde{v}\, \in\, [{\mathfrak g}\, ,{\mathfrak g}]^*$
defined as follows:
$$
\widetilde{v}(y) \,=\, h_0(y\, ,v)
$$
for all $y\, \in\, [{\mathfrak g}\, ,{\mathfrak g}]$.

The action of $G$ on $\mathfrak g$ preserves the decomposition
in \eqref{e1}, and the action of $G$ on ${\mathfrak z}_{\mathfrak g}$
is trivial.
Since $h_0$ is $K$--invariant, these imply that $\widetilde{v}$ is
left invariant by the action of $K$ on $[{\mathfrak g}\, ,{\mathfrak 
g}]^*$. We note that $[{\mathfrak g}\, ,{\mathfrak
g}]^*$ is identified with $[{\mathfrak g}\, ,{\mathfrak g}]$ using the
Killing form. There is no nonzero element of
$[{\mathfrak g}\, ,{\mathfrak g}]$ fixed by $G$ because
$[{\mathfrak g}\, ,{\mathfrak g}]$ is semisimple; since $K$ is Zariski
dense in $G$, this implies that
$$
[{\mathfrak g}\, ,{\mathfrak g}]^K\, =\, 0\, .
$$
In particular, $\widetilde{v}\,=\, 0$. Hence the decomposition in
\eqref{e1} is orthogonal with respect to $h_0$.

The natural projection
$$
f\, :\, [G\, ,G]\times Z_G\, \longrightarrow
\, G
$$
is a finite \'etale Galois covering. Since the decomposition in
\eqref{e1} is orthogonal, the $2$--form $f^*\omega_G$ decomposes as
\begin{equation}\label{e2}
f^*\omega_G\,=\, p^*_1 \omega_1 +p^*_2 \omega_2\, ,
\end{equation}
where $p_i$ is the projection of $[G\, ,G]\times Z_G$ to the
$i$--th factor, and $\omega_1$ (respectively, $\omega_2$) is the
$(1\, ,1)$--form on $[G\, ,G]$ (respectively, $Z_G$) associated
to the right translation invariant Hermitian metric obtained by
translating $h_0\vert_{[{\mathfrak g}\, ,{\mathfrak g}]}$
(respectively, $h_0\vert_{{\mathfrak z}_{\mathfrak g}}$). Form
\eqref{e2},
$$
f^*\omega^{\delta-1}_G\,=\, (p^*_1\omega^{\delta_1-1}_1)\wedge
p^*_2\omega^{\delta_2}_2 + (p^*_1\omega^{\delta_1}_1)\wedge
p^*_2\omega^{\delta_2-1}_2\, ,
$$
where $\delta_1$ and $\delta_2$ are the complex dimensions of
$[G\, ,G]$ and $Z_G$ respectively. Hence
\begin{equation}\label{e3}
f^*d \omega^{\delta-1}_G\,=\,
d f^*\omega^{\delta-1}_G\,=\, (p^*_1d\omega^{\delta_1-1}_1)\wedge
p^*_2\omega^{\delta_2}_2 + (p^*_1\omega^{\delta_1}_1)\wedge
p^*_2d \omega^{\delta_2-1}_2\, .
\end{equation}

Since $Z_G$ is abelian, it follows that $d \omega_2\,=\, 0$.
Therefore, from \eqref{e3} we conclude that $d\omega^{\delta-1}_G\, =\,
0$ if $d\omega^{\delta_1-1}_1\,=\, 0$. Therefore, it is enough
to prove the proposition for $G$ semisimple.

We assume that $G$ is semisimple.

Since the inner product $h_0$ is $K$--invariant, the Hermitian
structure $h_G$ is invariant under the left--translation action of
$K$ on $G$. Therefore, the element
\begin{equation}\label{e4}
(d\omega^{\delta-1}_G)(e) \,\in\, \wedge^{2\delta-1}
({\mathfrak g}\otimes_{\mathbb R}{\mathbb C})^*
\end{equation}
is preserved by the action of $K$ on $\wedge^{2\delta-1}
({\mathfrak g}\otimes_{\mathbb R}{\mathbb C})^*$
constructed using the adjoint action of $K$ on ${\mathfrak g}$.

The Killing form on $\mathfrak g$ produces a nondegenerate
symmetric bilinear form on ${\mathfrak g}\otimes_{\mathbb R}{\mathbb 
C}$. Using it, the $K$--module $\wedge^{2\delta-1}
({\mathfrak g}\otimes_{\mathbb R}{\mathbb C})^*$ gets identified
with ${\mathfrak g}\otimes_{\mathbb R}{\mathbb C}$. There is
no nonzero element of ${\mathfrak g}\otimes_{\mathbb R}{\mathbb C}$
which is fixed by $K$ because $G$ is semisimple and $K$ is Zariski
dense in $G$. In particular, the $K$--invariant element
$(d\omega^{\delta-1}_G)(e)$ in \eqref{e4} vanishes.
Since $d\omega^{\delta-1}$ is invariant under the right--translation
action of $G$ on itself, and $(d\omega^{\delta-1}_G)(e)\,=\, 0$,
we conclude that $d\omega^{\delta-1}_G\,=\, 0$.
\end{proof}

Let
$$
\Gamma\, \subset\, G
$$
be a closed discrete subgroup such that the quotient 
manifold $G/\Gamma$
is compact. The right--translation invariant Hermitian structure $h_G$
on $G$ descends to a Hermitian structure on $G/\Gamma$. This
Hermitian structure on $G/\Gamma$ will be denoted by $\widetilde{h}$.
Let $\widetilde{\omega}$ be the $(1\, ,1)$--form on $G/\Gamma$
defined by $\widetilde{h}$; so $\widetilde{\omega}$ pulls back to
the form $\omega_G$ on $G$. From Proposition \ref{prop1} we know that
\begin{equation}\label{e5}
d\widetilde{\omega}^{\delta-1}\,=\, 0\, .
\end{equation}

For a coherent analytic sheaf $E$ on $G/\Gamma$, define the
\textit{degree} of $E$
$$
\text{degree}(E)\, :=\, \int_{G/\Gamma} c_1(\det (E))\wedge
\widetilde{\omega}^{\delta-1}\, ;
$$
{}from \eqref{e5} it follows immediately that $\text{degree}(E)$
is independent of the choice of the first Chern form for
the (holomorphic) determinant line bundle
$\det (E)$; see \cite[Ch.~V, \S~6]{Ko} for determinant bundle.

For any $g\, \in\, G$, let
\begin{equation}\label{bg}
\beta_g\, :\, M\, :=\, G/\Gamma \, \longrightarrow\, G/\Gamma
\end{equation}
be the left--translation automorphism defined by $x\, \longmapsto\,
gx$. A coherent analytic sheaf $E$ over $G/\Gamma$
is called \textit{invariant} if for each $g\, \in\, G$, the pulled
back coherent analytic sheaf $\beta^*_g E$ is
isomorphic to $E$. Note that an invariant coherent analytic sheaf
is locally free.

\begin{theorem}\label{thm1}
Let $E$ be an invariant holomorphic vector bundle on $G/\Gamma$.
Then $${\rm degree}(E)\,=\, 0\, .$$
\end{theorem}

\begin{proof}
Since $E$ is invariant, it admits a holomorphic connection
\cite[Theorem 3.1]{Bi}. Any holomorphic connection on $E$
induces a holomorphic connection on the determinant line bundle
bundle $\det (E)\, :=\, \bigwedge^r E$, where $r$ is the
rank of $E$.

Let
$$
D\, :\, \det (E)\, \longrightarrow\, \det (E)\otimes
\Omega^{1,0}_{G/\Gamma}
$$
be a holomorphic connection on $\det (E)$; see \cite{At}
for the definition of a holomorphic connection. Let
$$
\overline{\partial}_{\det (E)}\, :\, \det (E)\, \longrightarrow\,
\det (E) \otimes\Omega^{0,1}_{G/\Gamma}
$$
be the Dolbeault operator defining the holomorphic structure on
$\det (E)$. Then $D+\overline{\partial}_{\det (E)}$ is a connection
on $\det 
(E)$. Let ${\mathcal K}(D+\overline{\partial}_{\det (E)})$ be the 
curvature of the connection $D+\overline{\partial}_{\det (E)}$. We
note that
$$
{\mathcal K}(D+\overline{\partial}_{\det (E)})\,=\, (D+
\overline{\partial}_{\det (E)})^2 \,=\, D^2\, ,
$$
because the differential operator $D$ is holomorphic, and 
$\overline{\partial}_{\det (E)}$ is integrable, meaning
$\overline{\partial}^2_{\det (E)}\,=\,  0$.
Therefore, ${\mathcal K}(D+\overline{\partial}_E)$ is a
differential form of $G/\Gamma$ of type $(2\, ,0)$. (In fact, the
form $D^2$, which is called the curvature of the holomorphic
connection $D$, is holomorphic, but we do not need it.)

As ${\mathcal K}(D+\overline{\partial}_E)$ is of
type $(2\, ,0)$, and $\widetilde{\omega}$ is of type $(1\, ,1)$,
$$
{\rm degree}(\det (E))\,=\, \int_{G/\Gamma}
{\mathcal K}(D+\overline{\partial}_E)\wedge 
\widetilde{\omega}^{\delta-1}\,=\, 0\, .
$$
Since $c_1(E)\,=\, c_1(\det (E))$, the theorem follows.
\end{proof}

A vector bundle $E$ over $G/\Gamma$ is called \textit{semistable}
if
$$
\frac{\text{degree}(V)}{\text{rank}(V)}\, \leq\,
\frac{\text{degree}(E)}{\text{rank}(E)}
$$
for every coherent analytic subsheaf $V\, \subset\, E$ of
positive rank.

\begin{lemma}\label{lem0}
Let $E$ be a torsionfree coherent analytic sheaf on $G/\Gamma$.
For any $g\, \in\, G$, 
$$
{\rm degree}(E)\, =\, {\rm degree}(\beta^*_g E)\, ,
$$
where $\beta_g$ is the map in \eqref{bg}.
\end{lemma}

\begin{proof}
Since $G$ is connected, the map $\beta_g$ is homotopic to the identity
map of $G/\Gamma$. Hence
$$
c_1(\det (\beta^*_g E)) - c_1(\det (E)) \,=\, d\alpha\, ,
$$
where $\alpha$ is a smooth $1$--form on $G/\Gamma$, and $c_1
(\det (\beta^*_g E))$ and $c_1(\det (E))$ are first Chern forms. Now,
$$
{\rm degree}(\beta^*_g E)-{\rm degree}(E)\,=\,
\int_{G/\Gamma} (c_1(\det (\beta^*_g E))-c_1(\det (E)))\wedge
\widetilde{\omega}^{\delta-1}
$$
$$
=\, \int_{G/\Gamma}
(d\alpha)\wedge \widetilde{\omega}^{\delta-1}\,=\,
\int_{G/\Gamma}\alpha\wedge d\widetilde{\omega}^{\delta-1}\, =\, 0
$$
by Proposition \ref{prop1}.
\end{proof}

\begin{lemma}\label{lem1}
Let $E$ be an invariant holomorphic vector bundle on
$G/\Gamma$. Then $E$ is semistable.
\end{lemma}

\begin{proof}
Let
$$
0\, =\, V_0\, \subset\, V_1\, \subset\, \cdots\, \subset\,
V_{\ell-1} \, \subset\, V_\ell \,=\, E
$$
be the unique Harder--Narasimhan filtration for $E$ \cite[p. 590,
Theorem 3.2]{Br}. We recall that
\begin{equation}\label{e6}
\frac{\text{degree}(V_1)}{\text{rank}(V_1)}\, >\,
\frac{\text{degree}(E)}{\text{rank}(E)}
\end{equation}
if $V_1\, \not=\, E$. From the uniqueness of $V_1$ and Lemma
\ref{lem0} it follows that for any $g\,\in\, G$ and
any isomorphism of $E$ with $\beta^*_g E$, the image of the
composition
$$
\beta^*_g V_1\, \hookrightarrow\, \beta^*_g E\,
\stackrel{\sim}{\longrightarrow}\, E
$$
coincides with $V_1$. This implies that $V_1$ is invariant.
Therefore,
$$
\text{degree}(V)\,=\, 0\, =\, \text{degree}(E)
$$
by Theorem \ref{thm1}. Now from \eqref{e6} we conclude that
$V_1\,=\, E$. Hence $E$ is semistable.
\end{proof}

Let $H$ be any affine complex algebraic group. Let $E_H$
be an invariant holomorphic principal $H$--bundle over $G/P$,
which means that the principal $H$--bundle $\beta^*_gE_H$ is 
holomorphically isomorphic to $E_H$ for all $g\, \in\, G$. From
Lemma \ref{lem1} we know that the adjoint vector bundle
$\text{ad}(E_H)$ is semistable. If $H$ is reductive then the
semistability of $\text{ad}(E_H)$ implies that the principal 
$H$--bundle $E_H$ is semistable; see \cite{Ra} for the
definition of semistable principal bundles.

It is a natural question to ask whether Lemma \ref{lem1} remains
valid if the reductive group $G$ is replaced by some more general
affine algebraic groups $G_1$.
The first step would be to construct a suitable Hermitian structure
on the compact quotient $G_1/\Gamma$, where $\Gamma$ is a cocompact
lattice in $G_1$. In order to be able to define semistability, 
the Hermitian structure on $G_1/\Gamma$ should satisfy the Gauduchon 
condition. It may be noted that the Hermitian structure on
$G_1/\Gamma$ given by a right--translation invariant Hermitian
structure on $G_1$ satisfies the Gauduchon condition (see
\cite[p. 74]{Bi0}. For a
general compact complex manifold equipped with a Gauduchon
metric, the degree of a line bundle with holomorphic connection
need not be zero; but it remains valid if the base admits a
K\"ahler metric \cite[p. 196, Proposition 12]{At}. The
compact complex manifold $G_1/\Gamma$ admits a K\"ahler metric
if and only if $G_1$ is abelian.


\end{document}